\documentclass[12pt]{article}
\usepackage{amscd}
\usepackage{amsmath,amssymb,mathrsfs}
\usepackage{hyperref}
\usepackage{xcolor}
\usepackage[all]{xypic}
\usepackage{mathtools}
\usepackage{esint}

\usepackage{times}
\usepackage{inputenc}
\usepackage{titlesec}
\usepackage{lipsum}
\newcommand\blfootnote[1]{%
  \begingroup
  \renewcommand\thefootnote{}\footnote{#1}%
  \addtocounter{footnote}{-1}%
  \endgroup
}

\def\inf{\operatorname{inf}}

 \usepackage{amsthm}
 \newtheorem{lemma}{Lemma}[section]
 \newtheorem{corollary}[lemma]{Corollary}
 
 \newtheorem{theorem}[lemma]{Theorem}
 
 \newtheorem{definition}[lemma]{Definition}
 \newtheorem{remark}[lemma]{Remark}

\advance\headheight1.15pt

\begin{document}
\title{\fontsize{14}{0}\selectfont Wiener-Lebesgue point property for Sobolev Functions on Metric Spaces
}
\author{\fontsize{11}{0}\selectfont 
M. Ashraf Bhat and G. Sankara Raju Kosuru$^*$ \\
\fontsize{10}{0}\textit{{Department of Mathematics, Indian
Institute of Technology Ropar, Rupnagar-140001, Punjab, India.}}
}
\date{}
\maketitle

\thispagestyle{empty}

\begin{abstract}
We establish a Wiener-type integral condition for first-order Sobolev functions defined on a complete, doubling metric measure space supporting a Poincar\'e inequality. It is stronger than the Lebesgue point property, except for a marginal increase in the capacity of the set of non-Lebesgue points.

\end{abstract}
    





\blfootnote{$^{*}$Corresponding Author} 
\blfootnote{G. Sankara Raju Kosuru: raju@iitrpr.ac.in} \blfootnote{%
Mohd Ashraf Bhat: ashraf74267@gmail.com}


\section{Introduction}

In a metric measure space $(X,d, \mu)$, a point $x \in X$ is said to be a Lebesgue point if, for any locally integrable function $u$ on $X$, the average value of the integral over balls centered at $x$ converges to $u(x)$ as the radii of the balls approach zero, {\it i.e.,} 

$$u(x)= \underset{r \to 0}{\lim} \frac{1}{\mu(B(x,r))}\int_{B(x,r)} u ~d\mu.$$

The  Lebesgue differentiation theorem in classical analysis (on Euclidean spaces) guarantees that almost every point of a locally integrable function defined on $\mathbb{R}^n$ is a Lebesgue point. If the functions are sufficiently regular, namely  they are Sobolev functions, then even more is true, {\it viz.}, the exceptional sets are of zero capacity. Besides, they exhibit a Lebesgue point property of the Weiner integral type \cite{Strong}, which is more potent than the classical Lebesgue point property. However, the exceptional sets have a slightly larger capacity. The aforementioned Weiner-Lebesgue point property on $\mathbb{R}^n$ is based on a classical result due to Meyers (see \cite[Theorem 2.1]{Meyers}).

Furthermore, the Lebesgue differentiation theorem extends to certain classes of metric spaces called the spaces of homogeneous type (see, for instance, \cite[p. 150]{Krantz}). The Sobolev functions on such metric spaces mirror the pointwise behaviour shown by them in the Euclidean case, {\it i.e.,} the sets of non-Lebesgue points are of capacity zero (see \cite[section 5.6]{MR2867756}, \cite[section 9.2]{Book} and \cite{LPP}). 


The central aim of this paper is to obtain a  Wiener-type integral condition for Sobolev functions on metric measure spaces. The Lebesgue point property is an implication of this condition. However, the sets of non-Lebesgue points provided by this condition are of $q-$capacity zero for all $1<q<p$, whereas in the actual Lebesgue point property, these sets are slightly smaller, {\it i.e.,} of $p-$capacity zero.  We closely follow the strategy used by Latvala in \cite{Strong}. To achieve our goals, we first establish several results concerning capacities, Sobolev functions, and their upper gradients in the setting of metric measure spaces. Our argument relies on several tools, such as fine topologies, capacities, density functions in geodesic metric spaces, etc. Additionally, we employ capacities corresponding to two different notions of Sobolev spaces on metric measure spaces. If $\mu$ is a finite measure,  we get a stronger integral condition, which is new even in Euclidean space.

\section{Preliminaries}
 In this paper, $X=(X,d,\mu)$ denotes a complete metric space equipped with a metric $d$ and a positive Borel regular measure $\mu$ such that $\mu(B) \in (0, \infty)$ for every (open) ball $B:=B(x,r) \subset X$, centered at some point $x \in X$ and radius $r>0$. Throughout, we assume that $\mu$ is a doubling measure, {\it i.e.,} there exists a constant $C_{\mu}>0$ such that for every ball $B \subset X$, $\mu(2B) \leq C_{\mu} \mu(B)$. The doubling property of $X$ ensures that it is separable \cite[p. 107]{Book}.
 
 Suppose that $B(x,R)$ is a ball in  $X$, $y \in B(x,R)$ and $0<r \leq R< \infty$. If there exist $Q>0$ and $C>0$ such that $$\frac{\mu(B(y,r))}{\mu(B(x,R))} \geq C\left( \frac{r}{R} \right)^Q,$$
 then we say that $X$ has relative lower volume decay of order $Q$ than $X$. It is easy to see that the doubling property of measure $\mu$ implies $X$ always has relative lower volume decay of some order  $Q \leq \log_2 \left(C_{\mu}\right)$ \cite[Lemma 3.3]{MR2867756}. The exponent $Q$ serves as a counterpart for the dimension of the space $X$. In particular, if $X$ is $n-$dimensional Euclidean space and $\mu$ is Lebesgue measure, then $Q=n$. Hereafter, $Q$ always refers to the order in the relative lower volume decay condition.

Let $1<p<\infty$ be fixed, unless otherwise specified. A measurable function  $u:X \to \mathbb{R}$ is said to be in $L^p(X)$ if $$\left\Vert u \right\Vert_{L^p(X)}:=\left(\int_X |u|^p d \mu \right)^{\frac{1}{p}}< \infty.$$
Functions in $L^p(X)$ are also referred to as $p-$integrable functions on $X$. We say $u \in L^p_{loc}(X)$ if $u \in L^p(A)$ for every compact subset $A \subset X$. If $u \in L^1_{loc}(X)$ and $A \subset X$ is measurable with  $0<\mu(A)< \infty$, then its integral average on A is denoted by $u_A$, {\it i.e.,}
$$u_A:=\fint_A u ~d\mu:=\frac{1}{\mu(A)}\int_A u ~d\mu.$$
Let $\alpha \in [0, \infty)$ and $R>0$. The centered fractional maximal function of $f\in L^1_{loc}(X)$ is defined by
$$\mathcal{M}_{\alpha,R}f(x):=\underset{r \in (0,R)}{\sup}r^{\alpha}\fint_{B(x,r)}|f|~d\mu.$$ If $\alpha=0$ and $R=\infty$, this coincides with the centered Hardy Littlewood maximal function.

There are several different ways to define Sobolev spaces on metric measure spaces. Here, we recall two commonly employed notions. This is because certain results pertinent to our discussion are available for one definition but not the other. However, we will later observe that these two notions are equivalent under our assumptions.

\subsection{Newton-Sobolev Spaces}

A curve of finite of length is called a rectifiable curve. We will only consider those curves which are non-constant, compact and rectifiable. Such a curve can be parameterized  by arc length $ds$. We say a property holds for $p-$almost every curve if it fails only for a curve family $\mathcal{C}$ with zero $p-$modulus, {\it i.e.,}  there exists a non-negative Borel $\rho \in L^p(X)$ such that $\int_{\gamma} \rho ~ds=\infty$ for every $\gamma \in \mathcal{C}$ \cite[Lemma 5.2.8]{Book}.
 
 The following notion of upper gradients due to Heinonen and Koskela \cite{Heinonen19981} serves as a substitute of the modulus of the usual gradient for functions defined in metric measure spaces. 

\begin{definition}
    A non-negative Borel function $g$ on $X$ is said to be an upper gradient of a real-valued function $u$ on $X$ if for all curves $\gamma:[0, l] \to X$,
\begin{equation}\label{Upper-gradient}
    |u(\gamma(0))-u(\gamma(l))| \leq \int_{\gamma}gds.
\end{equation}
\end{definition}
If (\ref{Upper-gradient}) holds for $p-$almost every curve, then $g$  is called a $p-$weak upper gradient of $u$.
 If $u$ has a $p-$weak upper gradient in $L^p_{loc}(X)$, then it has a minimal $p-$weak upper gradient therein. That is, there exists $g_u \in L^p_{loc}(X)$ satisfying $g_u \leq g$ {\it a.e.} for every $p-$weak upper gradient $g \in L^p_{loc}(X)$ of $u$  (see \cite[Theorem 7.16]{MR2039955}). 
 
 For any set $A \subset X$, we denote by $diam(A)$ the diameter of $A$.

 \begin{definition}
     The space $X$ is said to support a $p-$Poincar\'e inequality if for all balls $B \subset X$, all integrable functions $u$ on $X$ and all $p-$weak upper gradients $g$ of $u$, there exist constants $C>0$ and $\sigma \geq 1$ such that 
     $$\fint_B |u-u_B|d\mu \leq C~ diam(B) \left(\fint_{\sigma B} g^p d \mu \right)^{\frac{1}{p}},$$
     where $\sigma B$ is the ball having same centre as $B$ and radius $\sigma$ times that of $B$.
 \end{definition}
 We say $X$ is a $p-$Poincar\'e space if it supports a $p-$Poincar\'e inequality. Throughout the paper, we assume that $X$ is a $p-$Poincar\'e space. Let $\Tilde{N}^{1,p}(X)$ be the collection of all $p-$integrable functions $u$ on $X$ having a $p-$integrable upper gradient on $X$. We endow $\Tilde{N}^{1,p}(X)$ with the functional 
$$\left \Vert u \right \Vert_{\tilde{N}^{1,p}(X)}=\left( \int_X |u|^p~ d\mu + \underset{g}{\inf} \int_X g^p ~d\mu \right)^{\frac{1}{p}},$$ where the infimum is taken over all upper gradients $g$ of $u$.
It is important to point out that functions in $\Tilde{N}^{1,p}(X)$ are assumed to be defined everywhere and not just up to an equivalence class in the corresponding function space. The functional $\left \Vert \cdot \right \Vert_{\tilde{N}^{1,p}(X)}$ is only a seminorm. The Newton-Sobolev space $N^{1,p}(X)$ is obtained by passing to equivalence classes of functions in $\Tilde{N}^{1,p}(X)$, {\it i.e.,} 

 $$N^{1,p}(X):= \Tilde{N}^{1,p}(X)/{\sim},$$ where $u \sim v$ if and only if $\left\Vert u-v \right\Vert_{\Tilde{N}^{1,p}(X)}=0$.
We write $\left \Vert u \right \Vert_{N^{1,p}(X)}$ for the norm (quotient) of $u \in N^{1,p}(X)$. The space $N^{1,p}(X)$  with this norm is a Banach space. For a measurable set $E \subset X$, the space $\Tilde{N}^{1,p}(E)$ is defined by considering $(E,d|_E,\mu|_E)$ as a  metric measure space in its own right. 
 \begin{definition}

 The $p-$capacity of $E \subset X$ with respect to the space $N^{1,p}(X)$ is defined by 
 $$C_p(E):=\inf \left\Vert u \right\Vert^p_{N^{1,p}(X)},$$
 where the infimum is taken over all functions $u$ in $N^{1,p}(X)$ such that $u|_E \geq 1$.
      
 \end{definition}
 A property holds $p-$quasieverywhere ($p-$q.e.) if it holds everywhere except possibly on a set of $p-$capacity zero. The  Newton-Sobolev space with zero boundary values $N^{1,p}_0(E)$ is defined as the set of all functions $u \in N^{1,p}(X)$ for which $C_p(\{x \in X \setminus E~:~u(x) \neq 0 \})=0$ and is equipped with the norm 
 $$\left \Vert u \right \Vert_{N^{1,p}_0(E)}=\left \Vert u \right \Vert_{N^{1,p}(X)}.$$
 The space $N^{1,p}_0(E)$ with this norm is a Banach space.

 \begin{definition}
     Assume that $\Omega \subset X$ is bounded. The variational (relative) capacity of a set $E \subset \Omega$ is defined by
     $$cap_p(E,\Omega)=\underset{u}{\inf} \int_{\Omega} g^p_u d \mu,$$
     where the infimum is taken over all $u \in N^{1,p}_0(\Omega)$ such that $u \geq 1$ on E. If no such function $u$ exists, then $cap_p(E,\Omega)$ is  taken to be $\infty$.
 \end{definition}
 For more details on these topics, we refer our reader to the books \cite{MR2867756} and \cite{Book}.
 The following lemma \cite[Proposition 6.16]{MR2867756} relates the capacities $C_p$ and $cap_p$, and the measure $\mu$.
 \begin{lemma} \label{CapVSmeasure}
     Let $E \subset B=B(x,r)$ with $0<r< \frac{1}{8}diam(X)$. Then there exists a positive constant $C$ such that
     $$\frac{\mu(E)}{Cr^p} \leq cap_p(E,2B) \leq \frac{C_{\mu}\mu(B)}{r^p}$$ and
     $$\frac{C_p(E)}{C(1+r^p)}\leq cap_p(E,2B) \leq 2^p\left(1+\frac{1}{r^p} \right)C_p(E).$$
 \end{lemma}
\subsection{Haj{\l}asz-Sobolev Spaces}
 Haj{\l}asz, in \cite{MR1401074}, introduced an alternative approach to defining Sobolev spaces on metric measure spaces. This method is rooted in his pointwise characterization of Sobolev spaces defined on $\mathbb{R}^n$.
 Let $\mathcal{D}(u)$ be the set of all measurable functions  $g:X \to [0,\infty]$ such that 
 $$|u(x)-u(y)| \leq d(x,y) \left(g(x)+g(y) \right)$$
 for every $x,y \in X\setminus A$,  for some $A \subset X$ with $\mu(A)=0$. A function $g \in \mathcal{D}(u) \cap L^p(X)$ is called  a Haj{\l}asz gradient of $u$. 
 \begin{definition} 

 The space $M^{1,p}(X)$ consists of all functions $u\in L^p(X)$ for which there exists $g \in \mathcal{D}(u) \cap L^p(X)$. The space $M^{1,p}$ is equipped with norm
 $$\left \Vert u \right \Vert_{M^{1,p}(X)}= \left \Vert u \right \Vert_{L^p(X)}+ \underset{g}{\inf} \left \Vert g \right \Vert_{L^p(X)},$$
 where the infimum is taken over all Haj{\l}asz gradients $g$ of $u$.
     
 \end{definition}
 The space $M^{1,p}$ with the above norm is a Banach space. This space is known as Haj{\l}asz-Sobolev space. There is a natural capacity in Haj{\l}asz-Sobolev spaces. 
\begin{definition}
 The $p-$capacity of a set $E \subset X$ with respect to the space $M^{1,p}(X)$ is defined by
 $$C^{\prime}_p(E):=\underset{u \in \mathcal{A}(E)}{\inf} \left \Vert u \right \Vert^p_{M^{1,p}(X)},$$
 where $\mathcal{A}(E):=\{u \in M^{1,p}(X)~:~u \geq 1~\text{on an open neighbourhood of E}\}$.
 \end{definition}
  It must be noted that the capacities $C_p$ and $C^{\prime}_p$ are outer measures.  Since $X$ is a complete doubling metric measure space supporting a $p-$Poincar\'e inequality for $1<p< \infty$, it follows from \cite[Theorem 4.9]{Relation} and \cite[Theorem 12.3.9]{Book} that $M^{1,p}(X)$ and $N^{1,p}(X)$ are equal as sets. Moreover, their norms are equivalent (see \cite[Theorem 12.3.14]{Book}). Therefore, we do not make any distinction between them.

 \subsection{$p-$fine Topology}
 Assume that $X$ contains at least two distinct points. We recall some basics on fine topology associated with the space $N^{1,p}(X)$.
 \begin{definition}
    A set $A \subset X$ is said to be $p-$thin at $x \in X$  if there exists some $0<\delta<\infty$ such that
     $$\int_0^{\delta} \left(\frac{cap_p\left(A \cap B(x,r),B(x,2r)\right)}{cap_p\left(B(x,r),B(x,2r)\right)} \right)^{\frac{1}{p-1}}\frac{dr}{r}< \infty.$$
     A set $O \subset X$ is called $p-$finely open if $X\setminus O$ is $p-$thin at every $x \in O$.
     \end{definition}
     The integral in the above definition is termed as the Wiener integral. It appears  in the Weiner test for identification of regular boundary points \cite{Wiener}. Here, we use the convention that the integrand is $1$ if $cap_p\left(B(x,r),B(x,2r)\right)=0$. It is easy to verify that the collection of all $p-$finely open sets is a topology on $X$,  which we call the $p-$fine topology. A function $u:O \to \mathbb{R}$, defined on a $p-$finely open set $O$ is said to be $p-$finely continuous if it is continuous when $O$ is endowed with $p-$fine topology and $\mathbb{R}$ is endowed with the usual topology. For more details, refer to \cite[section 11.6]{MR2867756} and \cite{MR2373067}.
     
     We denote by $\Bar{E^p}$ the $p-$fine closure of a set $E \subset X$. The following result appears as first part of Lemma 4.8 in \cite{MR3361291}.
 
 \begin{lemma} \label{closure}
     Let $E \subset X$ be an arbitrary set. Then $C_p\left(\bar{E^p}\right)=C_p\left(E\right)$.
 \end{lemma}


 \section{Required Lemmas}

In this section, we establish a few technical results which will be used in the sequel. 
For this, we adopt the convention $a \lesssim b$ (or equivalently $b \gtrsim a$) to indicate the existence of a positive constant $C$, which may change from one occurrence to another, such that $a \leq C~b$.

 
 Let $r>0$ and $g \in L^p_{loc}(X)$ be a non-negative function. Consider the Riesz Potential
 $$J_{p,r}g(x):=\sum_{k=0}^{\infty} 2^{-k}r \left(\fint_{B(x,2^{-k}r)}g^p d \mu \right)^{\frac{1}{p}}.$$

 \begin{lemma}\label{Lemma2}
 Let $0<\delta<\infty$ and $g \in L^p(X)$  be a non-negative function. Then
 $$\int_X \left(\int_{0}^{\delta}\left(r^p\fint_{B(x,r)}g^pd \mu \right)^{\frac{1}{p-1}}\frac{dr}{r}\right)^{p-1}d \mu< \infty.$$
    
 \end{lemma}
 \begin{proof}
    For a fixed $x \in X$, setting $I_{\delta}(x):=\int_{0}^{\delta}\left(r^p\fint_{B(x,r)}g^pd \mu \right)^{\frac{1}{p-1}}\frac{dr}{r}$ and discretizing, we get
     
     \begin{eqnarray*}
         I_{\delta}(x) & = & \sum_{k=0}^{\infty} \int_{2^{-k-1}\delta}^{2^{-k} \delta}\left(r^p\fint_{B(x,r)}g^pd \mu \right)^{\frac{1}{p-1}}\frac{dr}{r} \\ & \lesssim & \sum_{k=0}^{\infty} \left( (2^{-k} \delta)^p \frac{1}{\mu(B(x,2^{-k-1}\delta))}\int_{B(x,2^{-k}\delta)}g^p d \mu\right)^{\frac{1}{p-1}}.
     \end{eqnarray*}
    Since $\mu$ is doubling,
    
    $$I_{\delta}(x) \lesssim \sum_{k=0}^{\infty} \left( 2^{-k} \delta \left(\fint_{B(x,2^{-k}\delta)}g^p d \mu \right)^{\frac{1}{p}}\right)^{\frac{p}{p-1}}.$$

   Now using the basic estimate $\sum |a_k|^{\theta} \leq \left( \sum |a_k|\right)^{\theta}$ for $\theta \geq 1$, we obtain
    $$I_{\delta}(x) \lesssim \left( J_{p, \delta}g(x) \right)^{\frac{p}{p-1}}.$$

    Therefore, by \cite[Lemma 10.4.8]{Book}, 
    \begin{eqnarray*}
     \int_X \left(I_{\delta}(x)\right)^{p-1} d \mu & \lesssim & \int_X \left( J_{p, \delta}g \right)^p d \mu \\ & \lesssim & \int_X g^p d \mu \\ & < & \infty.
    \end{eqnarray*}
    This completes the proof.
     
 \end{proof}
 Every function $u \in N^{1,p}(X)$ possesses a minimal $p-$weak upper gradient in $g_u \in L^p(X)$ which is unique up to sets of measure zero \cite[Theorem 2.5]{MR2867756}. Thus, as a consequence, we have the following:
 \begin{corollary} \label{cor3.2}
     Let $0<\delta<\infty$ and $u \in N^{1,p}(X)$. Then 
     $$\int_{0}^{\delta}\left(r^p\fint_{B(x,r)}g^p_ud \mu \right)^{\frac{1}{p-1}}\frac{dr}{r}< \infty~~~\text{for}~\mu-a.e. ~ x \in X.$$
 \end{corollary}
In the preceding result, the exceptional set possesses measure zero. However, when $1<p<n$ and $X=\mathbb{R}^n$, it is known that the exceptional set has $p$-capacity zero \cite[Lemma 2.4]{Strong}. For the metric setting, we get the following partial result.
 \begin{lemma}\label{Capacity}
     Let $0<\delta<\infty$, $u \in N^{1,p}(X)$ and $$E:=\left\{ x \in X~:~\int_{0}^{\delta}\left(r^p\fint_{B(x,r)}g^p_ud \mu \right)^{\frac{1}{p-1}}\frac{dr}{r}= \infty \right\}.$$ Then for every $q \in (1,p)$, $C_q(E)=0$.
 \end{lemma}
 \begin{proof} 
Without loss of generality, we take $\delta=1$. For any $x \in X$,
\begin{eqnarray*}
\int_{0}^{1}\left(r^p\fint_{B(x,r)}g^p_ud \mu \right)^{\frac{1}{p-1}}\frac{dr}{r} & = & \int_{0}^{1}r^{\frac{p-q}{p-1}}\left(r^q\fint_{B(x,r)}g^p_ud \mu \right)^{\frac{1}{p-1}}~\frac{dr}{r}
\\ & \leq & \sup_{0<r<1} \left(r^q\fint_{B(x,r)}g^p_ud \mu \right)^{\frac{1}{p-1}} \int_{0}^{1} r^{\frac{p-q}{p-1}-1} ~dr 
\\ & \leq & \frac{p-1}{p-q} \left( \mathcal{M}_{q,1}g^p_u(x)\right)^{\frac{1}{p-1}}.
\end{eqnarray*}
 Using monotonicity of the outer measure $C^{\prime}_q$ and applying \cite[Lemma 3.15]{MR1681586}, we get
\begin{eqnarray*}
 C^{\prime}_q(E) &  \leq &  \underset{N \to \infty}{\lim}C^{\prime}_q\left(\left\{ x \in X~:~ \left( \mathcal{M}_{q,1}g^p_u(x)\right)^{\frac{1}{p-1}}> \frac{p-q}{p-1}N \right\}\right) \\ & \lesssim & \underset{N \to \infty}{\lim}\frac{1}{N^{p-1}}\left\Vert g_u \right\Vert^p_{L^p(X)} \\ & = & 0.
\end{eqnarray*}
 From the definitions of the capacities, it follows that $ C_q(E) \leq C^{\prime}_q(E)$, and hence $ C_q(E)=0$.

 \end{proof}
For any function $u \in L^p_{loc}(X)$, we define $\Tilde{u}(x)$ pointwise for all $x \in X$ by the expression
 \begin{equation}\label{LD-eqn}
\Tilde{u}(x):= \underset{r \to 0}{\lim \sup}\fint_{B(x,r)}u~d\mu.    
 \end{equation}
 By the Lebesgue differentiation theorem, we get $\Tilde{u}(x) = u(x)$ $\mu-$a.e. on $X.$
 In the following lemma (see \cite[Theorem 5.3, Remark 5.21]{MR1681586}), we identify $u$ with $\Tilde{u}$ and discard the tilde notation.

 \begin{lemma}  \label{Holder QC}
     Let $u \in N^{1,p}(X)$ be defined pointwise everywhere by (\ref{LD-eqn}). Also let $1<p \leq Q$ and $0<\beta<1-\frac{1}{p}$. Then for every $\epsilon>0$, there is an open set $O$ with $C_{(1-\beta)p}(O)<\epsilon$ and a function $v$ such that $u=v$ everywhere in $X \setminus O$, $v \in N^{1,p}(X)$ and is H\"older continuous with exponent $\beta$ on every bounded subset of $X$.
 \end{lemma}

 \begin{lemma} \label{ball estimate on X}
     Let $1<p<Q$ and $1 \leq h \leq \frac{Qp}{Q-p}$ (or $p=Q$ and $h \geq 1$). Suppose $u \in N^{1,p}(X)$ and $g$ is a $p-$integrable upper gradient of $u$ on $X$. Then for any measurable set $S \subset X$, there exist constants $C>0$ and $\sigma>0$ such that
     $$\left(\frac{1}{\mu(B)}\int_{B \setminus S} |u|^h~d\mu \right)^{\frac{1}{h}}  \leq C \left( diam(B) \left( \frac{1}{\mu(B)} \int_{\sigma B} g^p~d\mu\right)^{\frac{1}{p}}+\left[\frac{\mu(B \setminus S)}{\mu(B)} \right]^{\frac{1}{h}}\left(\frac{1}{\mu(B)}\int_B|u|^h~d\mu \right)^{\frac{1}{h}} \right) .$$

 \end{lemma}
 \begin{proof}
     Let us first assume that $1<p<Q$ and $1 \leq h \leq p^{*}$, where $p^{*}:=\frac{Qp}{Q-p}$. By the H\"older inequality,
     $$\left(\fint_B|u-u_B|^h~d\mu \right)^{\frac{1}{h}} \leq \left(\fint_B|u-u_B|^{p^{*}}~d\mu \right)^{\frac{1}{p^{*}}}.$$
    Since $g|_B$ is a $p-$integrable upper gradient of $u|_B$, combining the above inequality with \cite[Theorem 8.3.2, Corollary 9.1.36]{Book}, we obtain
     \begin{equation} \label{E1}
     \left(\fint_B|u-u_B|^h~d\mu \right)^{\frac{1}{h}} \lesssim diam(B)\left(\fint_{\sigma B} g^p~d\mu \right)^{\frac{1}{p}}.
    \end{equation}
    Now by Minkowski's inequality
    \begin{equation} \label{E2}
    \left(\int_{B \setminus S}|u|^h~d\mu \right)^{\frac{1}{h}} \leq \left(\int_{B}|u-u_B|^h~d\mu \right)^{\frac{1}{h}}+\left(\mu(B \setminus S) \right)^{\frac{1}{h}}|u_B|.
    \end{equation}
    Again by an application of H\"older's inequality, we have
    \begin{eqnarray} \label{E3}
    \left(\mu(B \setminus S) \right)^{\frac{1}{h}}|u_B| & \leq & \left(\int_B|u|^h~d\mu \right)^{\frac{1}{h}}\left[\frac{\mu(B \setminus S)}{\mu(B)} \right]^{\frac{1}{h}}.
    \end{eqnarray}
    Combining (\ref{E1}), (\ref{E2}) and (\ref{E3}), we get the required estimate.

    Now assume $p=Q$ and $h \geq 1$. We can find unique $n \in \mathbb{N} \cup \{0\}$ such that $n< \frac{h(Q-1)}{Q} \leq n+1$. Therefore, for any $t \geq 0$,
    $$exp\left(t^{\frac{Q}{Q-1}}\right) \geq \frac{\left(t^{\frac{Q}{Q-1}}\right)^{n+1}}{(n+1)!} \geq \frac{t^h}{\left(\lceil{\frac{h(Q-1)}{Q}} \rceil+1 \right)!},$$
    where $\lceil{\cdot} \rceil$ is the ceiling function. Using the above observation in combination with \cite[Corollary 9.1.36]{Book}, we get 
    \begin{equation} \label{Case p=Q}
        \left(\fint_B|u-u_B|^h~d\mu \right)^{\frac{1}{h}} \lesssim diam(B)\left(\fint_{\sigma B} g^Q~d\mu \right)^{\frac{1}{Q}}.
    \end{equation}
    The result follows by combining (\ref{E2}), (\ref{E3}) and (\ref{Case p=Q}).
 \end{proof}

 \section{Main Theorem}

 \begin{theorem} \label{MT}
     Let $0<\delta<\infty$, $1<q<p$ and $u \in N^{1,p}(X)$ be defined as in Lemma \ref{Holder QC}. Then
     \begin{itemize} 
     \item[(i)] If $p<Q$ and $0 < h \leq \frac{Qp}{Q-p}$ (or $p=Q$ and $h>0$), then for $q-$q.e. $x \in X$ ,
     $$\int_0^{\delta}  \left(\fint_{B(x,r)}|u-u(x)|^h ~d\mu \right)^{\frac{p}{h(p-1)}} \frac{dr}{r}< \infty.$$
     
     
   \item[(ii)] If $p>Q$ and $h>0$, then
   $$\int_X\left(\int_0^{\delta}  \left(\fint_{B(x,r)}|u-u(x)|^h ~d\mu \right)^{\frac{p}{h(p-1)}} \frac{dr}{r} \right)^{p-1}~d\mu< \infty.$$
    \end{itemize}
 \end{theorem}
\begin{remark} Under the assumptions of the above theorem, the following integral condition is satisfied in each case,  $$\int_0^{1}  \left(\fint_{B(x,r)}|u-u(x)|^h ~d\mu \right)^{\frac{p}{h(p-1)}} \frac{dr}{r}< \infty~~~\text{for} ~q-q.e.~x \in X.$$
On discretization, it implies
$$\sum_{k=1}^{\infty}\left(\fint_{B(x,2^{-k})}|u-u(x)|^h ~d\mu \right)^{\frac{p}{h(p-1)}}< \infty~~~\text{for} ~q-q.e.~x \in X. $$
This gives us the Lebesgue point property
$$\underset{r \to 0}{\lim} \frac{1}{\mu(B(x,r))}\int_{B(x,r)}|u-u(x)|^h~d\mu=0~~q-q.e.~x \in X$$
\end{remark}
The preceding remark shows that, in a certain context, Theorem \ref{MT} is a stronger than the Lebesgue point property. However, it is important to note that the exceptional set obtained in this case is not as narrow as in the actual Lebesgue point property.
 \begin{proof}[Proof of Theorem \ref{MT}]
 
\begin{itemize}

 \item[(i)]  We first consider the case $p<Q$ and $0 < h \leq \frac{Qp}{Q-p}$.
By Lemma \ref{Holder QC}, for each $k \in \mathbb{N}$ there exist open sets $O_k \subset X$ with $C_q(O_k) <\frac{1}{2^k}$ and functions $v_k \in N^{1,p}(X)$ with $v_k=u$ everywhere in $X \setminus O_k$ such that $v_k$'s are H\"older continuous with  exponent $1-\frac{q}{p}$ on bounded subsets of $X$. 

 For a fixed $\delta>0$, let $E$ be as in Lemma \ref{Capacity} mand set
$ Z:= E \cup \underset{k \in \mathbb{N}}{\cap} \bar{O^q_k}$. Then, using Lemma \ref{Capacity} and Lemma \ref{closure}, we get that $C_q(Z)=0$. For a fixed $x \in X \setminus Z$, choose $k \in \mathbb{N}$ such that $x \in X \setminus \bar{O^q_k}$. Applying Lemma \ref{Holder QC}, we have, for any $r< \delta$,
\begin{equation} \label{HC-eqn}
|v_k(y) - v_k(x)| \lesssim \left( d(x,y) \right)^{1- \frac{q}{p}},
\end{equation}
where $y \in B(x,r) \cap V$ with $V:= \left(\bar{O^q_k} \right)^c$.  Since  $u(x)$ is finite $p-$q.e. $x \in X$ (\cite[Proposition 1.30]{MR2867756}), define $w=|u-u(x)|$ and notice that any upper gradient of $u$ is an upper gradient of $w$. Now decompose the integral by writing
\begin{equation} \label{split}
\left(\fint_{B(x,r)}w^h ~d\mu \right)^{\frac{p}{h(p-1)}}   \leq I_1 + I_2,
\end{equation}
where 
$$I_1:=\left( \frac{2}{\mu(B(x,r))} \int_{B(x,r) \setminus V} w^h ~d\mu \right)^{\frac{p}{h(p-1)}}~ \text{and}~~I_2:= \left(\frac{2}{\mu(B(x,r))}\int_{B(x,r) \cap V} w^h ~d\mu \right)^{\frac{p}{h(p-1)}}.$$

 Setting $\zeta_r:= \frac{\mu\left(B(x,r) \setminus V \right)}{\mu(B(x,r))}$ and $\theta_r:= \frac{diam(B(x,r))^p}{\mu(B(x,r))}$, assume first that $h \geq 1$. We estimate $I_1$ by using Lemma \ref{ball estimate on X} and obtain
\begin{eqnarray*} \label{estimate1}
I_1  & \leq & \left(2^{\frac{1}{h}}C \right)^{\frac{p}{p-1}} \left[ \left( \theta_r \int_{ B(x,\sigma r)} g_u^p~d\mu\right)^{\frac{1}{p}} +\left(\zeta_r \fint_{B(x,r)} w^h~d\mu \right)^{\frac{1}{h}} \right]^{\frac{p}{p-1}} \\ & \leq & \left(2^{1+\frac{1}{h}}C \right)^{\frac{p}{p-1}} \left[ \left( \theta_r \int_{B(x,\sigma r)} g_u^p~d\mu\right)^{\frac{1}{p-1}}+\left(\zeta_r\fint_{B(x,r)} w^h~d\mu \right)^{\frac{p}{h(p-1)}} \right].
\end{eqnarray*}
$I_2$ can be easily estimated using (\ref{HC-eqn}) as
\begin{eqnarray*}
I_2 & = &  \left(\frac{2}{\mu(B(x,r))}\int_{B(x,r) \cap V} |v_k(y)-v_k(x)|^h~d\mu \right)^{\frac{p}{h(p-1)}}  \\ & \leq & 2^{\frac{p}{h(p-1)}}  r^{\frac{p-q}{p-1}}.
\end{eqnarray*}
Therefore, from (\ref{split}), we have
\begin{equation*}
 \left(1- \left(2^{1+\frac{1}{h}} \zeta_r^{\frac{1}{h}} C \right)^{\frac{p}{p-1}} \right)\left(\fint_{B(x,r)}w^h ~d\mu \right)^{\frac{p}{h(p-1)}}  \lesssim  r^{\frac{p-q}{p-1}}+ \left( \frac{r^p}{\mu(B(x,r))} \int_{ B(x,\sigma r)} g_u^p~d\mu\right)^{\frac{1}{p-1}}
\end{equation*}
 As $\bar{O_k^q}$ is $q-$finely closed, $X \setminus \bar{O^q_k}$ is $q-$finely open. Therefore, $\bar{O_k^q}$ is $q-$thin at $x$. Applying Lemma \ref{CapVSmeasure}, we get
$$\int_0^{\delta} \left( \frac{\mu(B(x,r) \cap V)}{\mu(B(x,r))} \right)^{\frac{1}{p-1}}\frac{dr}{r}< \infty.$$
Since a complete and doubling metric space $X$ that supports a $p-$Poincar\'e inequality admits a geodesic metric that is biLipschitz equivalent to the underlying metric (see \cite[Corollary 8.3.16]{Book}), the density function $r \mapsto \frac{\mu(B(x,r) \cap V)}{\mu(B(x,r))}$ is continuous (\cite[Lemma 12.1.2]{Book}). Therefore, we must have
$$\underset{r \to 0^{+}}{\lim}~ \frac{\mu(B(x,r) \cap V)}{\mu(B(x,r))}=0.$$
 Choosing $r>0$ small enough such that $\zeta_r < \frac{1}{2^{h+1}C^h}$, we have 
\begin{equation} \label{limit}
\left(\fint_{B(x,r)}w^h ~d\mu \right)^{\frac{p}{h(p-1)}} \lesssim r^{\frac{p-q}{p-1}}+
\left( \frac{r^p}{\mu(B(x,r))} \int_{ B(x,\sigma r)} g_u^p~d\mu\right)^{\frac{1}{p-1}}.
\end{equation}
From H\"older's inequality, it follows that for sufficiently small $r>0$, (\ref{limit}) holds for any $0<h \leq \frac{Qp}{Q-p}$. The proof of $(i)$ is complete after applying Corollary \ref{cor3.2} and Lemma \ref{Capacity}.

 The case $p=Q$ and $h>0$ follows along similar lines.

\item[(ii)] Let $$I:= \int_X\left(\int_0^{\delta}  \left(\fint_{B(x,r)}|u-u(x)|^h ~d\mu \right)^{\frac{p}{h(p-1)}} \frac{dr}{r}\right)^{p-1}~d\mu.$$ 
     Applying \cite[Corollary 9.1.36(iii)]{Book}, for some
 $\sigma>0$ we have
     \begin{eqnarray*}
       I  & \lesssim & \int_X\left(\int_0^{\delta} \left( \fint_{B(x,r)} \left(diam(B(x,r))\left( \fint_{B(x,\sigma r)}g_u^p~d\mu\right)^{\frac{1}{p}}\right)^h ~d\mu \right)^{\frac{p}{h(p-1)}} \frac{dr}{r}\right)^{p-1}~d\mu \\ & \lesssim & 
       \int_X \left( \int_0^{\delta}  \left( r^p  \fint_{B(x,\sigma r)}g_u^p~d\mu \right)^{\frac{1}{p-1}} \frac{dr}{r} \right)^{p-1}~d\mu \\ & \lesssim & 
       \int_X \left(\int_0^{\frac{\delta}{\sigma}}  \left(r^p \fint_{B(x,r)}g_u^p~d\mu\right)^{\frac{1}{p-1}}\frac{dr}{r}\right)^{p-1}~d\mu.
     \end{eqnarray*}
     The result follows from Lemma \ref{Lemma2}. 
\end{itemize}  
 \end{proof}
\begin{remark}
    Suppose that the conditions of Theorem \ref{MT} are satisfied. If $\mu(X)<\infty$, it becomes evident that our proof leads to the stronger conclusion:
    $$\int_X\left(\int_0^{\delta}  \left(\fint_{B(x,r)}|u-u(x)|^h ~d\mu \right)^{\frac{p}{h(p-1)}} \frac{dr}{r} \right)^{p-1}~d\mu< \infty$$
    in each of the cases.
\end{remark}

\section*{Acknowledgements}
\begin{itemize}
\item[(i)] We would like to express our gratitude to Prof. Anders Bj\"{o}rn for his constructive comments and insightful suggestions, which greatly enriched the quality of this research.
\item[(ii)] We deeply appreciate the insightful comments and thorough evaluation provided by the reviewer, which significantly contributed to the improvement of our research paper.
\end{itemize}
\section*{Declarations}
%
%
\subsection*{Conflicts of Interest}
The authors declare that there is no conflict of interest.
\subsection*{Funding}
The first author (M. Ashraf Bhat) is supported by Prime Minister's Research Fellowship (PMRF) program (PMRF ID: 2901480).
%

\end{document}